\documentclass[a4paper,12pt]{article}
\usepackage{amsfonts}
\usepackage{amsthm}
\usepackage{amsmath}
\usepackage{latexsym}
\usepackage{amssymb}
\usepackage[ansinew]{inputenc}
\usepackage{graphicx}
\usepackage{dsfont}
\usepackage{authblk}
\newcommand{\nui}[1]{N(#1)}

\newtheorem{fed}{Definition}[section]

\newtheorem*{fed*}{Definition}
\newtheorem*{feds*}{Definitions}
\newtheorem{teo}[fed]{Theorem}
\newtheorem*{teo*}{Theorem}

\newtheorem{pro}[fed]{Proposition}
\theoremstyle{definition}
\newtheorem{rem}[fed]{Remark}

\newtheorem*{rems*}{Remarks}

\newtheorem{exa}[fed]{Example}

\def\coma{\, , \, }

\def\py{\peso{and}}
\newcommand{\peso}[1]{ \quad \text{ #1 } \quad }

\def\n0{n_{ \text{\rm \tiny o}}}

\def\suml{\sum\limits}

\def\bce{\begin{center}}
\def\ece{\end{center}}

\def\cD{\mathcal D}

\def\rk{\text{\rm rk}}
\def\noi{\noindent}
\def\cF{\mathcal F}
\def\cG{\mathcal G}

\def\EOE{\hfill $\triangle$}

\def\uno{\mathds{1}}

\def\bm{\left[\begin{array}}
\def\em{\end{array}\right]}
\def\ben{\begin{enumerate}}
\def\een{\end{enumerate}}
\def\bit{\begin{itemize}}
\def\eit{\end{itemize}}
\def\barr{\begin{array}}
\def\earr{\end{array}}

\def\la{\lambda}

\def\al{\alpha}
\def\N{\mathbb{N}}
\def\R{\mathbb{R}}

\def\C{\mathbb{C}}
\def\I{\mathbb{I}}

\def\cH{\mathcal{H}}
\def\cK{\mathcal{K}}

\def\cM{{\cal M}}

\def\ca{\mathbf{a}}
\def\bb{{\mathbf b}}

\def\da{^\downarrow}

 \DeclareMathOperator{\tr}{tr}

\DeclareMathOperator{\cc}{\mbox{co}}

\newcommand{\mat}{\mathcal{M}_d(\mathbb{C})}

\newcommand{\matsad}{\mathcal{H}(d)}

\newcommand{\matud}{\mathcal{U}(d)}

\newcommand{\matpos}{\mat^+}

\def\beq{\begin{equation}}
\def\eeq{\end{equation}}

\def\pausa{\medskip\noi}

\newcommand{\paren}[1]{\left(#1\right)}
\newcommand{\llav}[1]{\left\{#1\right\}}
\newcommand{\abs}[1]{\left|#1\right|}
\newcommand{\norm}[1]{\left\|#1\right\|}
\newcommand{\normdos}[1]{\left\|#1\right\|_2^2}

\newcommand{\aldd}{\cD(\alpha\coma \dd)}

\def\cc{\mathbf{c}}

\def\ca{\mathbf{a}}

\def\C{\mathbb{C}}
\def\I{\mathbb{I}}
\def\N{\mathbb{N}}
\def\R{\mathbb{R}}

\def\bb{{\mathbf b}}
\def\cc{{\mathbf c}}
\def\dd{{\mathbf d}}

\def\cD{\mathcal{D}}

\def\cF{\mathcal{F}}
\def\cG{\mathcal{G}}
\def\cH{\mathcal{H}}

\def\cK{\mathcal{K}}

\def\cM{{\cal M}}

\def\la{\lambda}
\def\al{\alpha}

\def\uno{\mathds{1}}

\def\ca{\mathbf{a}}

\def\EOE{\hfill $\triangle$}

\def\bdem{\begin{proof}}
	\def\edem{\end{proof}}

\def\beq{\begin{equation}}
\def\eeq{\end{equation}}

\def\bm{\left[\begin{array}}
	\def\em{\end{array}\right]}

\def\ben{\begin{enumerate}}
	\def\een{\end{enumerate}}

\def\bit{\begin{itemize}}
	\def\eit{\end{itemize}}

\def\barr{\begin{array}}
	\def\earr{\end{array}}

\def\bce{\begin{center}}
	\def\ece{\end{center}}

\def\noi{\noindent} 

\def\pausa{\medskip\noi}

\def\ua{^\uparrow}
\def\da{^\downarrow}

\def\coma{\, , \, }

\def\suml{\sum\limits}

\def\rk{\text{\rm rk}}

\begin{document}

\title{Best multi-valued approximants via multi-designs}
\author{Mar\'\i a Jos\'e Benac$\,^{a,b}$, Noelia Belén Rios$\,^{c,d}$, Mariano Ruiz$\,^{c,d}$\footnote{Partially supported by CONICET
(PIP 00954CO - 2022), UNLP (11X829), UNSE (23/C190-PIP-2022) 
 e-mail addresses: mjbenac@gmail.com, nbrios@mate.unlp.edu.ar, mruiz@mate.unlp.edu.ar}
}
\date{\vspace{.4cm} {\small $^a$Dto. Acad\'emico de Matem\'atica, FCEyT-UNSE, Santiago del Estero, Argentina \\ 
\vspace{.2cm}
$^b$Instituto de Recursos H\'\i dricos - FCEyT- CONICET, Santiago del Estero, Argentina\\ 
\vspace{.2cm}
$^c$Centro de Matemática de La Plata, FCE-UNLP, La Plata, Argentina\\ \vspace{.2cm}
$^d$IAM-CONICET, Buenos Aires, Argentina}}
\maketitle

\begin{abstract}
	Let  ${\mathbf d} =(d_j)_{j\in\mathbb{I}_m}\in \mathbb{N}^m$ be a decreasing finite sequence of positive integers, and let $\alpha=(\alpha_i)_{i\in\mathbb{I}_n}$ be a finite and non-increasing sequence of positive weights. Given a family  
	$\Phi^0=(\mathcal{F}_j^0)_{j\in\mathbb{I}_m}$  of Bessel sequences with $\mathcal{F}_j^0=\{f_{i,j}^0\}_{i\in \mathbb{I}_k}\in (\mathbb{C}^{d_j})^k$ for each $1\leq j\leq m$, our main purpose on this work is to characterize the best approximants of the $m$-tuple of frame operators of the elements of $\Phi^0$ in the set $D(\alpha,\mathbf d)$ of the so-called $(\alpha,\mathbf d)$-designs, which are the 
	$m$-tuples
	$\Phi=(\mathcal{F}_j)_{j\in\mathbb{I}_m}$ such that each $\mathcal{F}_j=\{f_{i,j}\}_{i\in\mathbb{I}_n}$ is a finite sequence in $\mathbb{C}^{d_j}$, 
	and 
	$\sum_{j\in\mathbb{I}_m}\|f_{i,j}\|^2=\alpha_i$ for $i\in\mathbb{I}_n$.
	Specifically, in this work we completely characterize the  minimizers of the Joint Frame Operator Distance (JFOD) function:
	$\Theta:D(\alpha,\mathbf d)\to \mathbb{R}_{\geq 0} $  given by
	$$\Theta(\Phi)=\sum_{j=1}^m \| S_{\mathcal{F}_j} - S_{\mathcal{F}^0_j}\|_2^2 \,,$$
	where $S_{\mathcal{F}}$ denotes the frame operator of $\mathcal{F}$ and $\|\cdot\|_2$ is the Frobenius norm.
	Indeed, we show that local minimizers of $\Theta$ are also global and we obtain an algorithm to construct the optimal $(\alpha,\mathbf d)$-desings. As an application of the main result, in the particular case that $m=1$, we also characterize global minimizers of a G-frames problem recently considered by He, Leng and Xu.
\end{abstract}


\section{Introduction}
Motivated by many applications in matrix theory, matrix approximation problems (or matrix nearness problems) have been studied for several years. There are many books and papers in the literature that deal with different variants of these problems, see for example \cite{gower} and \cite{Higham} for a more detailed discussion of the subject and references.

\pausa Let $\cM_{n\times d}(\C)$ be the space of complex matrices of size $n\times d$. Given a non empty subset $\mathcal{X}$ of $\cM_{n\times d}(\C)$ and $A\in \cM_{n\times d}(\C)$, a usual matrix nearness problem is to compute 
$$\delta= \min_{X\in \mathcal{X}} \, N(A-X)\,,$$
where $N(\cdot)$ is a unitary invariant norm, that is,  $N(UAW)=N(A)$ for every $A\in \cM_{n\times d}$ and every pair of unitary matrices $U\in \cM_n(\C)$ and $W\in \cM_d(\C)$. Typically, the matrix norm used is the Frobenius norm: $\|A\|_2=\tr(A^*A)$ which has some desirable properties.

\pausa If such a distance can be calculated, a natural issue that arises is to characterize the set of best approximants, that is the set
$$\mathcal{X}^{op}=\{X\in \mathcal{X} \, :\, N(A-X)=\delta\}.$$

\pausa Finite frame theory provided many of such matrix approximation problems (or Procustes type problems) related to frame designs.  Given a finite dimensional complex Hilbert space $\mathcal{H}$, a frame $\cF$ for $\mathcal{H}$ is simply a generating set of vectors of $\mathcal{H}$. Associated to a frame $\cF$, there is a positive definite bounded linear operator $S_\cF$ of $\mathcal{H}$, called {\it frame operator},   that allows to perform encoding-decoding schemes. For practical reasons sometimes it is useful to find frames with some structure whose frame operators are ``close'' to some definite positive operator $A$. These kind of approximation problems were considered by some of the authors in \cite{CMR} and \cite{MRiS}, in which they were attacked with various tools of matrix analysis, such as the Schur-Horn theorem or Lidskii inequalities. These results are also related to  optimal designs of frames with specific predetermined characteristics obtained by minimizing some convex potentials on  sets of frames (see \cite{MRS13}, \cite{mrs2},\cite{mrs3}).

\pausa  In \cite{MRiS1}, the authors solved completely a conjecture posed by N. Strawn in \cite{Strawn} related to an approximation problem. Given a  $S\in \matpos$ and a fixed finite sequence of positive weights $\alpha=(\alpha_i)_{i\in\I_m}$, N. Strawn considered the following setting: let
$\cD(\alpha,d)$ denote the finite sequences $\cF=\{f_i\}_{i\in\I_n}\in(\C^d)^n$ such that $\|f_i\|^2=\alpha_i$, $i\in\I_n$ . Consider in  $\cD(\alpha,d)$  the product metric (i.e. the metric as a subset of $(\C^d)^n)$; let $\Theta:\cD(\alpha,d)\rightarrow \R_{\geq 0}$, be given by $\Theta(\cF)=\|S-S_\cF\|_2$, where $\|\cdot\|_2$ denotes the Frobenius norm. Strawn conjectured that local minimizers of $\Theta$ where actually global minimizers.  This assertion becomes relevant in applied situations in which numerical methods based  on gradient descent or alternating projections methods are used to obtain local minimizers of $\Theta$
, \cite{liu}.

\pausa In \cite{MRiS1}, Strawn's conjeture was settled in the affirmative and the spectral and geometrical structures of the minimizers of the function $\Theta$ defined above (called the {\it frame operator distance}) were explicitly computed.  

\pausa In this work we consider a natural extension of the previous problem to a simultaneous approximation problem. Now, given a positive integer $m$, we consider the sets of $m$-tuples 
 $\Phi=(\cF_j)_{j\in \I_m}$,  such that each $\cF_j=\{f_{ij}\}_{i\in \I_n}$ is a sequence in $\C^{d_j}$ such that
$$\sum_{j\in  \I_m} \|f_{ij}\|^2=\alpha_i.$$
These $m$- tuples are called $(\alpha,\dd)$-designs. Then, given a fixed sequence of positive operators $\{S_j\}$, we consider the function
$$\Theta(\Phi)=\sum_{j=1}^m \| S_j - S_{\cF_j}\|_2^2 \,,$$
\noindent which measure the joint frame operator distance between  $\{S_j\}_{j\in \I_m}$ and the frame operators  $\{S_{\cF_j}\}_{j\in \I_m}$.

\pausa  The problem we consider in this work is to find the  $(\alpha,\dd)$-designs $\Phi^{\rm op}$ that minimize $\Theta$, which result in the best simultaneous approximation of
$S_j$, for $1\leq j\leq m$.
Moreover, 
since  the set of $(\alpha,\dd)$-designs can be endowed with a natural (product) metric, we also consider the study of the spectral and geometric structure of the local minimizers of $\Theta$ in this set.
Notice that the particular case $m=1$ 
represent Strawn's problem  described above.

\pausa The case  $m>1$ is original and correspond to a natural extension of Strawn's problem, so it is forseeable that similar techniques allow us to find an spectral and geometric characterization of the local and global minimizers of $\Theta$ in the set of $(\al,\dd)$-designs. Specifically, we solve the multivalue Strawn's problem through a translation of the multi-completion problem given in \cite{BMRS2},  which means that the minimum in $\Theta$ are attained in the $(\alpha,\dd)$-designs that minimize the joint convex potential for a suitable multi-completion problem.

\pausa These notes are organized as follows. In Section 2 we include some preliminaries about matrix analysis and $(\alpha,\dd)$-designs. In Section 3 we prove the main result, that local minimizers of $\Theta$ are global, and we obtain an spectral characterization of this minimizers. In Section 4 we present an algorithm to find (effectively) the best approximants among $(\alpha, \dd)$-designs. Finally, in Section 5, we apply the $m=1$ case to an approximation problem for G-frames (see for example \cite{sun0} and \cite{Sun}), considered in \cite{MLX}. This allows us to fully describe the minimizers for the distance problem considered and to suggest an algorithm that will allow us to construct the optimal G-frames.

\section{Preliminaries and notation}

In this Section we recall the notion of $(\alpha\coma \dd)$-design, the multi-completions and the main problems 
considered in \cite{BMRS2}, that plays a key role in our work. Next, we describe some basic notation and notions 
used throughout the rest of the paper.

\pausa
We let  $\mat$ for the algebra of $d\times d$ complex matrices. We denote by $\matsad\subset \mat$ the real subspace of selfadjoint matrices and by $\matpos\subset \matsad$ the cone of positive semidefinite matrices. We let $\matud\subset \mat$ denote the group of unitary matrices.
For $d\in\N$, let $\I_d=\{1,\ldots,d\}$ and let $\uno_d=(1)_{i\in\I_d}\in\R^d$ be the vector with all its entries equal to $1$.

\pausa
Given $x=(x_i)_{i\in\I_d}\in\R^d$ we denote by $x\da=(x_i\da)_{i\in\I_d}$ (respectively $x\ua=(x_i\ua)_{i\in\I_d}$) 
the vector obtained by rearranging the entries of $x$ in non-increasing (respectively non-decreasing) order. We 
denote by $(\R^d)\da=\{x\da:\ x\in\R^d\}$, $(\R_{\geq 0}^d)\da=\{x\da:\ x\in\R_{\geq 0}^d\}$ and analogously for $(\R^d)\ua$ and $(\R_{\geq 0}^d)\ua$. 

\pausa
We also denote by $I_d \in \cM_d(\C)$ the identity matrix. Given $S\in\mat$ we let $R(S)\subset\C^d$ denote the range (or image) of $S$ and $\rk(S)$ denote the rank of $S$, i.e. the dimension of $R(S)$. 
Given a matrix $A\in\matsad$ we denote by $\la(A)=\la\da(A)=(\la_i(A))_{i\in\I_d}\in (\R^d)\da$ the eigenvalues of $A$ counting multiplicities and arranged in non-increasing order, and by $\la\ua(A)$ the same vector  but arranged in non-decreasing order. On the other hand, 
we denote by $\sigma(A)\subset \R$ its spectrum, i.e. the set of eigenvalues of $A$. 
If $x,\,y\in\C^d$ we denote by $x\otimes y\in\mat$ the rank-one matrix given by $(x\otimes y) \, z= \langle z\coma y\rangle \ x$, for $z\in\C^d$.
\pausa

\subsection{Finite frames}
\pausa Given a finite sequence $\cF=\{f_i\}_{i\in\I_n}$ in $\C^d$, $S_\cF\in \matpos$  (a {\it Bessel sequence} using frame terminology) will denote the
{\it frame operator} of $\cF$, which is given by 
$${S_{\cF}}\,  f=\sum_{i\in\I_n} \langle f,f_i\rangle
f_i=\sum_{i\in\I_n} (f_i\otimes f_i)\, f \peso{for} f\in\C^d\,.$$

\noi If 
 there exists a constant $a>0$ such that
\beq\label{frame}
a\norm{f}^2\leq\sum_{i\in\I_n} \abs{\langle f,f_i\rangle}^2
\peso{for all} f\in\C^d\,.
\eeq
we say that $\cF$ is a {\it frame} for $\C^d$. This condition is equivalent to say that $\cF$ spans $\C^d$
or that $S_\cF$ is a positive
invertible operator acting on $\C^d$.

\pausa Recall now the notion of majorization between real vectors, which is a partial pre-order relation in $\R^d$ that arises naturally in matrix analysis, and that will play a central role throughout our work.
Let  $x,y\in\R^d$. We say that $x$ is
{\it submajorized} by $y$, and write $x\prec_w y$,  if
$$
\suml_{i=1}^j x^\downarrow _i\leq \suml_{i=1}^j y^\downarrow _i \peso{for every} 1\leq j\leq d\,.
$$
If $x\prec_w y$ and $\tr x = \sum_{i\in \I_d}x_i=\sum_{i\in \I_d} y_i = \tr y$,  then $x$ is
{\it majorized} by $y$, and write $x\prec y$. In addition, we say that $x$ is {\it strictly} majorized by $y$ if 
$x\prec y$ and $x\da  \neq y\da$.

\noi For convenience, we extend the definition to allow comparing vectors of positive entries and different sizes, if $x\in \R_{\geq 0}^k$ and $y\in \R_{\geq 0}^d$, we note $x\prec_w y$ if
$$
\suml_{i=1}^j x^\downarrow _i\leq \suml_{i=1}^j y^\downarrow _i \peso{for every} 1\leq j\leq \min\{k,d\}\,.
$$
and $x$ is majorized by $y$ if $x\prec_w y$ and $\sum_{i\in \I_k} x_i=\sum_{i\in \I_d} y_i$.

\noi In several applications of finite frame theory, it is important to construct families
$\cF=\{f_i\}_{i\in\I_k}\in (\C^d)^k$ in such a way that the frame operator $S_\cF$ and the squared norms $(\|f_i\|^2)_{i\in\I_k}$
are prescribed in advance. This problem is known as the frame design problem, and its solution can be obtained in 
terms of the Schur-Horn theorem for majorization.
\begin{teo}[See \cite{Illi}]\label{teo SH para marcos}
	Let $S\in \matpos$ and let $\ca=(a_i)_{i\in\I_k}\in (\R_{>0}^k)$. Then, the following statements are equivalent:
	\ben
	\item There exists $\cF=\{f_i\}_{i\in\I_k}\in (\C^d)^k$ such 
	that $S_\cF=S$ and $\|f_i\|^2=a_i\,$, for $i\in\I_k\,$;
	\item $\ca\prec \la(S)$.	
	\qed\een
\end{teo}
\subsection{Preliminaries on $(\alpha\coma\dd)$-designs}
Given a $m$-tuple of natural numbers $\dd=(d_1,d_2,\ldots,d_m)$, arranged in a non-increasing order, a $\dd$-design  is any  family of Bessel sequences:
$$\Phi=\{\cF_j\}_{j\in \I_m},$$
such that each $\cF_j=\{f_{i,j}\}_{i\in \I_k}$ is a Bessel sequence for $\C^{d_j}$. 

\pausa
Our interest is to consider $\dd$-designs with some restriction on the sizes of the vectors on the Bessel sequences.

\pausa
 Namely, let $\alpha=(\alpha_1,\alpha_2,\ldots,\alpha_n)\in (\R_{>0}^n)$ be a sequence of weights. Then, an $(\alpha,\dd)$-design $\Phi=\{\cF_j\}_{j\in \I_m}$ is a $\dd$-design such that
	$$\suml_{j\in\I_m}\| f_{ij}\|^2=\alpha_i, \, \text{ for }\, i\in\I_n\,.
	$$ 

\pausa	
The set of all $(\alpha\coma \dd)$-designs shall be	denoted by $\mathcal D(\alpha\coma \dd)$. Also, with the aim to simplify some calculations, we assume that the weights are arranged in a non-increasing order.

\pausa
Notice that, if $m=1$ and $\dd=d$, $(\alpha\coma \dd)$-designs generalize the notion of the structured Bessel sequences for $\C^d$ with prescribed norms given by $\alpha$. That is, those Bessel sequences $\cF=\{f_i\}_{i\in \I_n}$ whose vectors lie in the $\alpha$-torus
\beq\label{eq defi Balfad0}
\mathcal B_{\al\coma d}=\{ \cF=\{f_i\}_{i\in\I_n}\in (\C^d)^n:\ \|f_i\|^2=\al_i \, , \ i\in\I_n\}\, .
\eeq

 \pausa 
In \cite{BMRS2}, the authors studied the problem of finding $(\alpha, \dd)$-designs  $\Phi^{op}=\{\cF^{op}_j\}_{j\in \I_m}$ that complete an initial $\dd$-design $\Phi^0=\{\cF^0_j\}_{j\in \I_m}$ in an optimal sense.

\pausa
In what follows, we will detail this multi-completion problem and the results obtained in \cite{BMRS2}, which will be useful in the next Section.

\pausa
Consider an $n$-tuple $\alpha=(\alpha_i)_{i\in\I_n}\in \R_{>0}^n$, arranged in a non-increasing order and  let $\dd =(d_j)_{j\in\I_m}\in(\N^m)\da$ be such that $d_1\leq n$ (this last condition is to assure that the optimal completions are frames for their respective spaces).

\pausa
We shall consider the set of $(\alpha, \dd)$-designs, $\mathcal D(\alpha\coma \dd)$, endowed with the metric
\[m(\Phi,\Phi')=\sum_{j\in \I_m}\ \left(\, \sum_{i\in\I_n}\|f_{i,j}-f_{i,j}'\|^2\,\right )^{1/2}\,,\]
where $\Phi, \Phi'\in \mathcal D(\alpha\coma \dd)$. 

\pausa	
We set a fixed $\dd$-design $\Phi^0=(\cF_j^0)_{j\in\I_m}$. 
 The goal is to find and to  characterize optimal (multi) completions of $\Phi^0$ among the $(\alpha, \dd)$-designs. That completion is obtained by appending to each $\cF_j^0$ the vectors of the respective Bessel sequence in the $(\alpha, \dd)$-design. 

\pausa
Here, the optimality is measured in terms of (joint) Benedetto-Fickus potential ${\rm P}$ of the multi-completions. That is, the goal is to find the local minimizers of the function $\Psi:\cD(\alpha\coma \dd)\rightarrow \R_{\geq 0}$, given by 
\begin{equation}\label{potencial conjunto}
\Psi (\Phi)={\rm P}(\Phi^0\coma \Phi) =\sum_{j\in\I_m} \tr(S^2_{(\cF_j^0\coma\cF_j)} )
=\sum_{j\in\I_m}\sum_{i\in\I_{d_j}} \la^2_i(S_{(\cF_j^0\coma\cF_j)})
 \, ,
 \end{equation}
where
$S_{(\cF_j^0\coma\cF_j)}=S_{\cF_j^0}+S_{\cF_j}$ denotes the frame operator of the sequence $(\cF_j^0\coma\cF_j)\in (\C^{d_j})^{k+n}$, for $j\in\I_m$ and the metric in $\cD(\alpha \coma \dd)$ is induced by the distance defined above.

\pausa
Now we are able to present a summarized version of the main result of \cite{BMRS2} that shall be useful in the sequel.

\pausa
First, let  $\la_j=(\la_{i,j})_{i\in\I_{d_j}}=\la\ua(S_{\cF_j^{0}})\in (\R^{d_j}_{\geq 0})\ua$, for $j\in\I_m$ be the vectors of eigenvalues (arranged in a non-decreasing order) of each frame operator $S_{\cF_j^0}$.

\pausa

\begin{teo}[\cite{BMRS2}]\label{teo: prelims opt multicomp}
There exist vectors $\nu_j\in \R^{d_j}$ such that, 
for $\tilde \Phi\in \cD(\alpha\coma \dd)$, we have
$$\tilde\Phi  \ \text{ is  a local minimizer of } \ \Psi \ \text{ on }\  \cD(\alpha\coma \dd)\quad \iff\quad \lambda(S_{(\cF^0_j\coma \tilde \cF_j)})=\nu_j\da\,.$$
Moreover,
\ben 
\item $\tilde \Phi$ is a global minimizer of $\Psi$.
\item For $j\in\I_m$, $S_{\tilde \cF_j}$ commutes with $S_{\cF_j^0}$ and $S_{\tilde \cF_j}+ S_{\cF_j^0}$ is invertible. In particular, $(\cF^0_j\coma \tilde \cF_j)$ is a frame for $\C^{d_j}$.
\een
\end{teo}
\pausa
The results proved in \cite{BMRS2} state a stronger feature fulfilling local minima: the frame operators $S_{\tilde \cF_j}$ not only commute with $S_{\cF_j^0}+S_{\tilde \cF_j}$ but also the vectors  $\tilde \cF_j=\{\tilde f_{ij}\}$ are eigenvectors of $S_{\cF_j^0}+S_{\tilde \cF_j}$.  As a consequence, $\tilde \cF_j$ decomposes into mutually orthogonal sets of vectors for each $j\in \I_m$. 

\pausa 
These results allow to describe the spectra $\nu_j$ as $\nu_j=\max \{\cc\coma \la_j\}$ (entry-wise maximum) where $\cc\in (\R^{d_1})^\downarrow$ is a vector constructed from the data $\alpha$ and $\la_j$, $j\in \I_m$.

\section{Local minimizers for the joint frame operator distance (JFOD)}
In this section we will present a simultaneous approximation problem for Bessel sequences that generalizes previous results shown in \cite{MRiS1}. Taking as a starting point a $\dd$-design $\Phi^0=\{\cF_j^0\}_{j\in \I_m}$ the goal is to characterize $(\alpha,\dd)$-designs that are local minimizers for some distance function defined on the frame operators. The approach is similar to the one developed in \cite{MRiS1}: it reduces to finding the local (global) minima of a suitable joint convex potential for the $(\alpha,\dd)$-design  problem  described in the previous section for a particular case of initial data.

\pausa
Given two $\dd$-designs $\Phi^1=\{\cF_j^1\}_{j\in \I_m}$ and $\Phi^2=\{\cF_j^2\}_{j\in \I_m}$, whose sequences of frame operators are $\Sigma_{\Phi^1}=\{S_{\cF_j^1}\}_{j\in \I_m}$ and $\Sigma_{\Phi^2}=\{S_{\cF_j^2}\}_{j\in \I_m}$, respectively, we define their joint frame operator distance (JFOD)  as follows:
\[\text{dist}_{JFOD}(\Sigma_{\Phi^1}\coma \Sigma_{\Phi^2})=\left( \sum_{j\in \I_m}\|S_{\cF_j^1}-S_{\cF_j^2}\|_2^2\right)^{\frac{1}{2}}\]
where the norm $\|\cdot\|_2$ is the Frobenius norm.

\pausa
Let $\Phi^0=\{\cF_j^0\}_{j\in \I_m}$ be a (fixed) $\dd$-design and consider a set of weights $\alpha$ as in previous section. As it was announced, our objective is to characterize those $(\alpha\coma \dd)$-designs that best approximate $\Phi^0$ in terms of the JFOD.

\pausa
In order to properly pose the problem to study, we define the function to minimize:

\begin{fed}
	\label{defi gen fod}\rm Let $\Phi^0=\{\cF_j^0\}_{j\in \I_m}$ be a $\dd$-design for $m\in \mathbb{N}$ and $\al=(\al_i)_{i\in\I_n}\in(\R_{>0}^n)\da$. Consider the  function  
	$$\Theta:\aldd\to \R_{\geq 0} \peso{given by} \Theta(\Phi)=\text{dist}^2_{JFOD}(\Sigma_{\Phi^0}\coma \Sigma_{\Phi})= \sum_{j\in \I_m}\|S_{\cF_j^0}-S_{\cF_j}\|_2^2, $$
for $\Phi=\{\cF_j\}_{j\in \I_m} \in \cD(\alpha\coma \dd)$.  
\end{fed}
\bigskip

\pausa
Next proposition shows that we can restate the JFOD problem as a multi-completion problem,  as in the case $m=1$ studied in \cite{MRiS1}. 

\pausa
\begin{rem}\label{traslacion}
Given an initial (fixed) $\dd$-design $\Phi^0=\{\cF_j^0\}_{j\in \I_m}$, with frame operators $\Sigma_{\Phi^0}=\{S_{\cF_j^0}\}$, take $M=\max_{j\in \I_m} \|S_{\cF_j^0}\|$ and choose any $\dd$-design $\tilde \Phi^0=\{\tilde \cF_j^0\}_{j\in \I_m}$ such as $\Sigma_{\tilde \Phi^0}=\{M\cdot I_{d_j}-S_{\cF_j^0}\}_{j\in \I_m}$. It is clear that such a $\dd$-design exists since $M\cdot I_{d_j} -S_{\cF_j^0}\in \cM_{d_j}^+(\C)$ for every $j\in \I_m$.
\end{rem}
\begin{pro}\label{traduction}
	Consider $\Phi^0$, $\tilde \Phi^0$ as in Remark \ref{traslacion}. Then,  
$\Phi=\{\cF_j\}_{j\in \I_m} \in \aldd$ is a local minimizer of $\Theta$  if and only if it is a local minimizer of $\Psi$ in $\aldd$, given  the initial data $\tilde \Phi$,  $\al$ and the strictly convex function $\varphi(t)=t^2$.
\end{pro}  
\bdem

Given $\Phi=\{\cF_j\}_{j\in \I_m} \in \cD(\alpha\coma \dd)$ and $M$ as in Remark \ref{traslacion},
$$S_{\cF_j^0}- S_{\cF_j}=M\cdot I_{d_j}+(S_{\cF_j^0}-M\cdot I_{d_j})-S_{\cF_j}=M\cdot I_{d_j}-(S_{\tilde \cF_j^0}+ S_{\cF_j}),$$
thus,
\begin{eqnarray*}
\Theta(\Phi)&=&\sum_{j\in \I_m} \| S_{\cF_j^0} - S_{\cF_j}\|_2^2
=\sum_{j\in \I_m}\tr([M\cdot I_{d_j}-(S_{\tilde \cF_j^0}+ S_{\cF_j}) ]^2)\\
&=&\sum_{j\in \I_m}[M^2 d_j- 2M\tr(S_{ \tilde \cF_j^0}+ S_{\cF_j})]+\sum_{j\in \I_m}\tr([S_{\tilde \cF_j^0}+ S_{\cF_j}] ^2)\\
&=&
k+\sum_{j\in \I_m}\tr([S_{\tilde \cF_j^0}+ S_{\cF_j}]^2)=k+\Psi(\Phi)\,.
\end{eqnarray*}
Where  $k=\sum_{j\in \I_m} M^2 d_j- 2M\tr(S_{\tilde \cF_j^0}+ S_{\cF_j})$ denotes a constant.
Hence $\Theta(\Phi)=\Psi(\Phi)+k$ for every $\Phi \in \cD(\alpha\coma \dd)$. In particular, local minimizers of
$\Theta$ and $\Psi$ (with their respective initial data) must coincide. 
\edem

\pausa
As a consequence, the complete characterization of local (global) minimizers for $\Theta$ can be carried out by the results summarized in Theorem \ref{teo: prelims opt multicomp}:

\pausa
Let $\tilde{\nu}_j$, $j\in \I_m$, denote the  spectra  of $S_{(\tilde \cF_j^0 \coma \cF_j)}$, where $\Phi=\{\cF_j\}_{j\in \I_m}$ is a local (global) minimizer of $\Psi$, for the initial data given by the $\dd$-design $\tilde \Phi^0$ constructed from $\Phi^0$ as it was described in  Remark \ref{traslacion}. Notice that, if we denote by $\la_j=\la\da (S_{\cF_j^0})$ the vector of eigenvalues of $S_{\cF_j^0}$ arranged in non-increasing order, then  the eigenvalues $\tilde{\la}_j=\la\ua(S_{\tilde \cF_j^0})=M \uno_{d_j}-\la_j \in (\R^{d_j}_{\geq 0})\ua$.

\begin{teo}\label{multi strawn} 
		Let $\delta_j=(\delta_{i,j})_{i\in\I_{d_j}}= M \,\uno_{d_j}-\tilde{\nu}_j$. Then,
	\begin{enumerate}
		\item
		$\displaystyle \min\llav{ \ \Theta(\Phi)=\sum_{j\in \I_m} \normdos{ S_{\cF_j^0} - S_{\cF_j}}:\ \Phi \in \cD{(\al,\dd)} \ }=\sum_{j\in \I_m} \|\delta_j\|^2$.

		\item If  $\Phi \in \cD{(\al,\dd)}$, then
		$$\Theta(\Phi)=\sum_{j\in \I_m} \normdos{ S_{\cF_j^0} - S_{\cF_j}}=\sum_{j=1}^m \|\delta_j\|^2 \peso{if  and only if,} \la_j(S_{\cF_j^0}-S_{\cF_j})=\delta_j\da\,.$$ 
		In this case, there exists an onb $\{v_{i,j}\}_{i\in\I_{d_j}}$ of $\C^{d_j}$ (for each $j\in\I_m$) such that
		\beq \label{eq teo rep opt str}
		S_{\cF_j^0} =\sum_{i\in\I_{d_j}} \la_{i,j}\ v_{i,j}\otimes v_{i,j} \py S_{\cF_j}=\sum_{i\in\I_{d_j}} (\la_{i,j}-\delta_{i,j})\ v_{i,j}\otimes v_{i,j}\,.
		\eeq
		\item  Local minimizers of $\Theta$ in $\aldd$ are also global minimizers.
		
	\end{enumerate}

\end{teo}

\bdem
The assertions of the statement are consequences of the Proposition \ref{traduction} and Theorem \ref{teo: prelims opt multicomp}. 
 In fact,  by Proposition \ref{traduction} local minima for $\Theta$ in $\aldd$ coincide with local (and hence global) minimizers for $\Psi$, where the initial data is given by $\tilde \Phi^0$.
 
 \pausa
 In particular, $\Phi=\{\cF_j\}_{j\in \I_m}$ is a local minimizer for $\Theta$ in $\aldd$ if and only if, for each $j\in \I_m$, the spectrum of $$S_{(\tilde \cF_j^0 \coma \cF_j)}=S_{\tilde \cF_j^0}+S_{ \cF_j}=M\cdot I_{d_j}-S_{\cF_j^0}+S_{ \cF_j}$$ is given by the vector $\tilde \nu_j$, that characterize the minimizers of $\Psi$, according Theorem \ref{teo: prelims opt multicomp}. In particular, the spectra of $S_{\cF_j}-S_{\cF_j^0}$ are given by the vectors $\delta_j$, $\forall j\in \I_m$, which implies
 $$\|S_{\cF_j}-S_{\cF_j^0}\|_2^2=\|\delta_j\|^2.$$
 
 \pausa
Moreover, by the same result, $S_{\cF_j}$ commutes with $M\cdot I_{d_j}-S_{\cF_j^0}$, so   $$S_{\cF_j}\cdot S_{\cF_j^0}=S_{\cF_j^0}\cdot S_{\cF_j}.$$
Therefore, $S_{\cF_j}$ and $S_{\cF_j^0}$ can be simultaneously diagonalized as in \eqref{eq teo rep opt str}.

\edem

\section{Explicit computation of $\delta=(\delta_j)_{j\in \I_m}$}
Once we established in Prop. \ref{traduction} the link between the approximation problem with optimal  $(\alpha, \dd)$- multi-completions, we can compute the vectors $\delta_j$, $j\in \I_m$, by reinterpreting the description of the optimal spectra $\nu_j$ done in \cite{BMRS2} using the ``translated'' initial data.

\pausa
Consider the notation introduced in the previous section, so that $\la_j=(\la_{i,j})_{i\in\I_{d_j}}\in (\R^{d_j}_{\geq 0})\da$ denote the spectrum of each $S_{\cF_j^0}$ and $M=\max_{j\in \I_m} \|S_{\cF_j^0}\|$. Then, for the construction of the vectors $\tilde \nu_j$ we shall use 
\beq \label{lambda con y sin tilde} 
\tilde \la_j=M\cdot \uno_{d_j}-\la_j=(M-\la_{i,j})_{i\in\I_{d_j}}\in (\R^{d_j}_{\geq 0})\ua,
\eeq i.e. the vector of eigenvalues of $S_{\tilde \cF_j^0}=M\cdot I_{d_j}-S_{\cF_j^0}$, counted with multiplicities and arranged in non decreasing order,  along with the weights $\alpha=(\alpha_i)_{i\in \I_n}$.

\pausa
According the results shown in \cite{BMRS2}, there is a unique vector
 $\cc\in (\mathbb{R}^{d_1})\da$, computable from $\{\tilde \la_j\}_{j\in \I_m}$ and  $\alpha$, such that each spectrum $\tilde \nu_j$ of the optimal completion $S_{\tilde \cF_j^0}+S_{\cF_j}$ is described as
\beq \label{el nu como max}
\tilde \nu_j=\max ((\cc)_{d_j}, \tilde \la_j)\,,
\eeq
where the maximum is taken entry-wise and  $(\cc)_{d_j}$ is the truncation of $\cc$ on its $d_j$ first entries

\pausa
The construction of the vector $\cc$ is done with some detail in \cite{BMRS2}. Mainly, it can be characterized as the unique (up to rearrangements) vector in $\mathbb{R}^{d_1}$ such that,
if 
\beq \label{el c} 
\cc=(c_1\uno_{s_1}, c_2\uno_{s_2},\cdots, c_p\uno_{s_p})\,,
\eeq
where $c_1>c_2>\cdots >c_p>0$ and $\sum_{k=1}^p s_k=d_1$,
then each $c_k$  and $s_k$ satisfy
\beq \label{la constante de b trasladada 1}(\alpha_i)_{i=i_{k-1}+1}^{i_k}\prec \paren{\sum_{j:\, i\leq d_j} (c_k-\tilde \la_{i,j})^+}_{i=i_{k-1}+1}^{i_k}=\paren{\sum_{j:\, i\leq d_j} ( \la_{i,j}-(M-c_k))^+}_{i=i_{k-1}+1}^{i_k}
\eeq
and
\beq \label{la constante de b trasladada 2} (\alpha_i)_{i=i_{p-1}+1}^n\prec \paren{\sum_{j:\, i\leq d_j} (c_k-\tilde \la_{i,j})^+}_{i=i_{p-1}+1}^{d_1}=\paren{\sum_{j:\, i\leq d_j} (\la_{i,j}-(M-c_k))^+}_{i=i_{p-1}+1}^{d_1}
\eeq
\noi
for $i_0=0$ and $i_j=\sum_{i=1}^j s_i$.

 \pausa
 Since for each $j\in \I_m$,  $\delta_j=M\uno_{d_j}-\tilde \nu_j$, where $\tilde \nu_j$ can be constructed as before, we are able to propose an algorithm that computes $\delta_j$ from the previous assertions.
 
 \pausa
First, let $\bb$ be the vector in $(\mathbb{R}^{d_1})\ua$ defined as $\bb=M\cdot \uno_{d_1}-\cc$. Then, by equations \eqref{lambda con y sin tilde}, \eqref{el nu como max}  and the characterization of $\delta_j$:
\beq \label{los nus}
\delta_j=\min ( (\bb)_{d_j}, \la_j)\,.
\eeq
From the equations \eqref{la constante de b trasladada 1} and \eqref{la constante de b trasladada 2} that define  the vector $\cc$  we  construct
$$\bb=(b_1\uno_{s_1},b_2\uno_{s_2},\cdots, b_p\uno_{s_p})\in (\mathbb{R}^{d_1})\ua\,,
$$
inductively as follows (as before, we let $i_0=0$ and $i_j=\sum_{k=1}^j s_k$): 

	\begin{fed}\rm\label{los ind y las cons}

Let us suppose that we have found the indices $i_0=0<i_1<\cdots <i_k$ (therefore we have $s_1=i_1$, $s_j=i_j-i_{j-1}$, for $j=1,\ldots ,k$) and the constants $b_1<b_2<\cdots <b_k$.

\pausa
Then, we define $i_{k+1}$ and $b_{k+1}$ as:
	\begin{equation}\label{el sk}
	i_{k+1}:=\max\llav{i \, :\, i_k+1\leq i \leq d_1: \, (\al_i)_{i=i_k+1}^{i^*}\prec\paren{\beta_{i,k+1}}^{i}_{i=i_{k}+1}}\quad\text{for}\,\,\, 0\leq k\leq p-1\,, 
\end{equation}
where $i^*$ and $\beta_{i,k+1}$ are determined from the following cases:
	
	\pausa
	Case 1)	If $i\in\I_{d_1-1}$ then $i^*=i$, and
	$$
	\paren{\beta_{i,k+1}}^{i}_{i=i_{k}+1}=\paren{\sum_{j:\, i\leq d_j} (\la_{i,j}-b_{k+1,i}^*)^+}_{i=i_{k}+1}^{i}
	$$ 
	and $b_{k+1,i}^*$, for $i\geq i_k+1$, is the unique solution of the equation 

	$$
	\sum_{i=i_k+1}^{i} \al_i=\sum_{i=i_k+1}^i\,\,
	\sum_{j:\,  i\leq d_j}(\la_{i,j}-x)^+\,.
	$$

	\pausa
	Case 2) If $i=d_1$, then $i^*=n$, and  
	$$
	\paren{\beta_{i,k+1}}^{d_1}_{{{i=i_{p-1}+1}}}=\paren{\sum_{j:\, i\leq d_j} (\la_{i,j}-b_{k+1,d_1}^*)^+}_{i=i_{p-1}+1}^{d_1}
	$$ 
	and $b_{k+1,d_1}^*$, is the unique solution of the equation 
	
	$$
	\sum_{i=i_{p-1}+1}^{n} \al_i=\sum_{i=i_{p-1}+1}^{d_1}
\,\,
	\sum_{j:\, i\leq d_j}(\la_{i,j}-x)^+\,.
	$$

We denote $b_{k+1}=b^*_{k+1,i_{k+1}}$	for $k\in \I_{p-1}$.
\EOE
\end{fed}

\pausa
This algorithm that produces the vector $\bb$, which generates $\delta=\{\delta_j\}_{j\in \I_m}$ is deduced from the characterization of $\cc$ given in \cite{BMRS2}. We omit the proof of the correct ordering in the entries of $\bb$ (that is, that the $b_{k+1}$ produced in this way is such that $b_k<b_{k+1}$)  and of the uniqueness of $\bb$ since they follow directly from the results proved in \cite{BMRS2}.
\begin{rem}\rm
Note that the  construction of $\delta$ proposed in Definition \ref{los ind y las cons} does not depend on the translation parameter $M$. This means that the algorithm that produces $\delta$ only requires as initial data the spectra $\la_j$ of the frame operators $S_{\cF_j^0}$ (with multiplicities and arranged in non decreasing order)  and the set of weights $\alpha=(\alpha_i)_{i\in \I_n}$, as expected.

\end{rem}

\begin{exa}
	In the following example, we implement the described algorithm for the same initial data considered in \cite[Example 5.5]{BMRS2}.
	
	\pausa
	That is: for  $d=(7,5,3)$ consider $\Phi^0=\{\cF_j^{0}\}_{j\in\I_3}$ be given
	in matrix form by

	\[\cF_1^{0}=\left[\begin{array}{rrrr}
		0.3066&    1.6919&   -1.14&    0.0488\\
		0.9339&   -0.4353&   -0.2197&    0.2354\\
		-1.8151&    0.8134&    0.3742&    0.2428\\
		1.7690&    1.0168&    0.8745&   -0.045\\
		-0.4706&    0.7223&    0.8595&    0.0609\\
		1.1678&   -0.0164&    0.0839&    0.2206\\
		-0.1574&    0.48&    0.042&   -0.3589
	\end{array}
	\right]
	\]
	\[\cF_2^{0}=\left[\begin{array}{rrr}
		-2.723&   -0.068&     -0.5242\\
		-2.2341&   -0.5975&    0.2401\\
		-1.5660&    0.7992&    0.0219\\
		2.2048&   -0.1835&   -0.4038\\
		0.5298&    0.2569&    0.0631
	\end{array}
	\right]
	\]
	and
	\[\cF_3^{0}=\left[\begin{array}{rrr}
		-0.8048&   -0.9958&   -0.1026\\
		1.0153&   -0.5127&   -0.4653\\
		0.5669&   -0.4955&    0.6877
	\end{array}
	\right]
	\]
	In this case, the spectra of $S_{\cF^0_j}$ are: 
	\[ \la_1=(9,\, 5.5,\, 3,\, 0.3,\, 0,\, 0,\, 0), \qquad \la_2=(20,\,1.1,\,0.5,\,0,\,0) \quad \text{ and } \; \la_3=(2,\,1.5,\,0.7).\]
	Consider the set of weights $\al=(40,\;   35,\;    9,\;    5,\;    4.5,\;    3,\;    2.4,\;   2)$. An implementation of the previously discussed algorithm produces
	\[\bb=20\cdot \uno_7-\cc=(-5.9833,\,   -5.9833,\,   -2.3778,\,   -2.3778,\,   -2.3778,\,   -2.3778,\,   -2.3778),\]
	so $\delta_1=\bb$,\quad  $\delta_2=(\bb)_5$\;  and \;  $\delta_3=(\bb)_3$. Thus, the minimal value for the multi-approximation is
	\[\|\delta_1\|^2+\|\delta_2\|^2+\|\delta_3\|^2=265.685\,.\]
	
	\pausa
	Moreover, by applying well-known algorithms that allow the construction of matrices with prescribed  spectra and  column norms (see for example \cite{Dill}), we obtain the following  solution to the multi-approximation problem:
	\[\tilde \cF_1=\left[\begin{array}{rrrrrrrr}
		-1.8371&    1.6608&    1.2662&    0.5185&    0.3026&    0.3396&    0.1226&    0.5687\\
		-0.5473&   -1.2442&    0.7725&    0.2088&   -0.5880&   -1.0094&   -0.4046&   -0.0201\\
		1.0932&    2.3814& -0.0088&   -0.7383&   -0.8395&   -0.1102&   -0.3142&   -0.2439\\
		-2.7000&   -0.2801&   -0.9373&   -0.8292&   -0.2106&   -0.4169&    0.2434&   -0.1652\\
		-0.1785&   1.1942&   -1.1371&    0.9572&   -0.5076&   -0.4684&    0.0565&   -0.3832\\
		-1.1612&   -0.9604&   -0.0092&    0.1824&   -0.6634&    0.9483&   -0.6628&   -0.5109\\
		-0.2751&    0.6683&   -0.0512&   -0.0647&    1.0101&   -0.3055&   -0.8896&   -0.7003
	\end{array}
	\right]
	\]
	
	\bigskip
	
	\[\tilde \cF_2=\left[\begin{array}{rrrrrrrr}
		-2.3145&   -2.0750&   -1.2187&    0.0102&   -0.2369&   -0.4077&   -0.1708&   -0.1559\\
		-1.0396&   -2.7753&    0.4457&    0.3734&    0.5631&    0.3362&    0.4060&    0.3706\\
		-2.6610&   0.4670&    0.04&   -0.2241&    0.5553&    0.2283&    0.4003&    0.3654\\
		2.2525&    1.2076&   -0.9958&   -0.1239&    0.6724&    0.0069&    0.4848&    0.4425\\
		0.0638&    0.8863&   -0.1217&    1.4798&   -0.0001&   -0.0314&   -0.0001&   0
	\end{array}
	\right]
	\]
	and
	\[\tilde \cF_3=\left[\begin{array}{rrrrrrrr}
		-0.1356&    2.7412&   -0.1706&   -0.04&   -0.0834&   -0.0451&   -0.0601&   -0.0549\\
		-2.2991&   -0.3745&   -0.7734&   -0.1812&   -0.3781&   -0.2047&   -0.2726&   -0.2488\\
		-1.5759&    0.1557&    1.1430&    0.2679&    0.5587&    0.3025&    0.4028&    0.3677
	\end{array}
	\right]
	\]
\end{exa}

\section{An application to a distance problem for G-frames} 

In the previous section we generalized the 
approximation problem studied in \cite{MRiS1}, to the  setting of $(\alpha,\dd)$-designs, that is, as a simultaneous approximation to a family of semi-definite positive matrices with some structured matrices, that come from these $(\alpha,\dd)$-designs.

	\pausa
	In this section we study another natural generalization, to the set of G-frames, that was posed and studied in \cite{MLX}. So, we briefly recall the concept of G-frames, introduced by W. Sun in \cite{sun0}.

	\pausa
	A family $\cF=\{T_i\}_{i\in I}$ of linear bounded operators $T_i$ from $\C^d$ to an analysis space $\C^n$ is a G-frame for $\C^d$ if there exist constants $a,b>0$ such that
$$a\|x\|^2\leq \sum_{i\in I} \, \|T_ix\|^2\leq b\|x\|^2, $$
for every $x\in \C^d$. If only the upper inequality holds, we say that $\cF$ is a G-Bessel sequence for $\C^d$.

	\pausa
Given a G-Bessel sequence  $\cF=\{T_i\}_{i\in I}$, its frame operator $S_\cF$ is defined as 
$$S_\cF=\sum_{i\in I} T_i^*T_i.$$

	\pausa
Let $\alpha=(\alpha_i)_{i\in I_m}$ be a non increasing finite sequence of positive weights. Consider the set
$$\Lambda_\alpha=\{\cF=\{T_i\}_{i\in \I_m}: \cF \; \text{ is a G-Bessel sequence for }\mathcal{H}\; \text{  with  }\; \|T_i\|_2^2=\alpha_i\}.$$

	\pausa
Let $A$ be a positive semi definite operator of $\mathcal{H}$. Our goal is to study the following approximation problem:

	\pausa
Compute 
\begin{equation}\label{problema G-frames}
\min_{\cF\in \Lambda_\alpha} \|A-S_\cF\|_2,
\end{equation}
and characterize the G-Bessel sequences that attain the
minimum distance.

	\pausa
We shall see that this problem can be treated as a particular case of the problem considered in \cite{MRiS1, MRiS}.

	\pausa
First, we need the following characterization of the frame operators of elements in $\Lambda_\alpha$. Recall that the dimension of the  analysis space that we are considering is $n$.
\begin{pro}\label{caracterizacion de operadores en Lambda alfa}
	Let $S\in \cM_{d}(\C)^+$ with eigenvalues given by $\la \in \R_{\geq 0}^d$ and let  $\alpha=(\alpha_i)_{i\in \I_m}\in (\R_{>0}^m)\da$. Then, there exists $\cF\in \Lambda_\alpha$ with $S_\cF=S$ if and only if
	\beq \label{mayo larga} (\frac{\alpha_1}{n}\uno_n\coma\frac{\alpha_2}{n}\uno_n\coma\cdots\coma \frac{\alpha_m}{n}\uno_n)\prec \la \,.
	\eeq 
\end{pro} 
\begin{proof}
	On one direction, if \eqref{mayo larga} holds, the Schur-Horn theorem implies the existence of a (vector) frame $\cF_{vec}=\{f_j\}_{j\in \I_{nm}}$ such that $\|f_{j}\|^2=\frac{\alpha_i}{n}$, for $(i-1)n+1\leq j\leq in$, and whose frame operator $S_{\cF_{vec}}$  is $S$. Let $T^*\in L(\C^{nm},\mathcal{H})$ be the bounded linear operator such that $T^*e_i=f_i$, where $\{e_j\}_{j\in \I_{nm}}$ is the standard orthonormal basis in $\C^{nm}$.  
		
	\pausa	
Consider a (fixed) orthonormal basis $\{b_j\}_{j\in \I_n}$ for $\mathcal{K}$.  Let $W_i\in L(\C^{nm},\mathcal{K})$ be the partial isometry defined such that
$$W_ie_{(i-1)n+j}=b_j, \; \text{ for } j=1,\ldots,n\quad \text{ and } \quad 	W_ie_k=0\;  \text{ otherwise.}$$
That is, it implies that $W_iW_i^*=I_n$ and $W_i^*W_i=P_i$, where $I_n$ is the identity in $\cK$ and $P_i$ is the diagonal projection of $\C^{mn}$ onto the subspace generated by
$$J_i=\{e_{(i-1)n+j}\,:\, j=1,\ldots,n\}.$$
Thus, $\sum_{i\in \I_m}P_i=I_{mn}$, the identity in $\C^{mn}$.

\pausa	Notice that, if we set $\cF=\{T_i\}_{i\in \I_m}$ where $T_i\in L(\mathcal{H},\mathcal{K})$ is defined such that $T_i=W_iT$, for $i\in \I_m$, then 
	$\cF\in \Lambda_\alpha$ and $S_\cF=S_{\cF_{vec}}=S$.
	
\pausa Indeed, 
$$S_{\cF}=\sum_{i\in \I_m}T_i^*T_i=\sum_{i\in \I_m}T^*W_i^*W_iT=T^*\left(\sum_{i\in \I_m}W_i^*W_i\right)T=T^*\left(\sum_{i\in \I_m}P_i\right)T=T^*T=S_{\cF_{vec}}$$
and
$$\|T_i\|_2^2=\tr(T_i^*T_i)=\tr(TT^*P_i)=\sum_{j\in \I_{mn}}\|T^*P_ie_j\|^2=\sum_{j\in \I_{n}}\|f_{(i-1)n+j}\|^2=\alpha_i.$$
	
		\pausa
	On the other side, suppose that $\cF=\{T_i\}_{i\in \I_m}$ is a G-Bessel sequence in $\Lambda_\alpha$ such that  $S_\cF=S$. Then,  
	$$\alpha_i=\|T_i\|_2^2=\tr (T_i^*T_i)\qquad \text{for all}\quad i\in\I_m\,,$$
	in particular, if we denote $f_{ij}=T^*_ib_j$, we have
	$$\frac{\alpha_i}{n}\uno_n\prec (\|f_{i1}\|^2\coma \|f_{i2}\|^2\coma \cdots\coma \|f_{in}\|^2)\prec \la(T_i^*T_i).$$
	Here, the majorization on the left holds since it is easy to see that for every $x\in  \R_{\geq 0}^d$, $(\frac{\tr{x}}{d})\uno_d\prec x$, while the comparison on the right is due to Theorem \ref{teo SH para marcos}.
	
	\noi Again by use of Schur-Horn theorem we deduce from the previous majorization relationship that there is, for each $i\in \I_m$, a Bessel sequence $G_i=\{g_{ij}\}_{j\in \I_d}$ for $\mathcal{H}$ such that $\|g_{ij}\|^2=\frac{\alpha_i}{d}$ and  such that $S_{G_i}=T_i^*T_i$.
	
\pausa Define the linear operator $T^*\in L(\C^{mn},\mathcal{H})$ by $T^*b_{(i-1)n+j}=g_{ij}$, for $i\in \I_m$ and $j\in \I_n$.

\pausa In particular, $T^*P_iT=T_i^*T_i$, using the previous definition for the orthogonal projections $P_i$. Then, 
$$T^*T=T^*\left(\sum_{i\in \I_m}P_i\right)T	=\sum_{i\in \I_m}T_i^*T_i=S.$$

		\pausa
	Therefore, the sequence $\cG=\{G_i\}_{i\in \I_m}$, constructed by juxtaposition is a Bessel sequence for $\mathcal{H}$, with synthesis operator $T^*\in L(\C^{mn},\mathcal{H})$ and frame operator $S_{\cG}=T^*T=S$.
	
\pausa	Finally, since the squared norms of the elements in $\cG$ are given by the vector $(\frac{\alpha_1}{n}\uno_n\coma\frac{\alpha_2}{n}\uno_n\coma\cdots\coma \frac{\alpha_m}{n}\uno_n)$, by Schur-Horn theorem we conclude that
	$$(\frac{\alpha_1}{n}\uno_n\coma\frac{\alpha_2}{n}\uno_n\coma\cdots\coma \frac{\alpha_m}{n}\uno_n)\prec \la(S_\cG)=\la.$$
	
\end{proof}

	\pausa
We are now in a position to prove the existence and characterization of approximants of the problem posed in Eq. \eqref{problema G-frames}.

\begin{teo}\label{g-frame global} 
	Let $\alpha=(\alpha_j)_{j\in\I_m}(\R_{\geq 0}^m)\da$, and consider the set of G-frames $\Lambda_\alpha$ as before. Let $A\in \cM_{d}(\C)^+$. Then, there exists
	$\cF^{\rm op}=\{T_j^{\rm op}\}_{j\in \I_m}\in \Lambda_\alpha$
	such that 
	\beq
	\|A- S_{\cF^{\rm op}}\|_2\leq \|A- S_{\cF}\|_2\quad\text{for all}\quad \cF\in \Lambda_\alpha\,.
	\eeq
	Moreover, the minimal distance (and the approximants)  can be computed using the spectrum of $A$ and the weights  $(\frac{\alpha_1}{n}\uno_n\coma\frac{\alpha_2}{n}\uno_n\coma\cdots\coma \frac{\alpha_m}{n}\uno_n)$ as the initial data for the classical approximation problem.
\end{teo}

\begin{proof}
	Given a $\cF\in \Lambda_\alpha$, by Proposition \ref{caracterizacion de operadores en Lambda alfa} there is a Bessel sequence $\cF_{vec}$ for $\mathcal{H}$ such that $S_\cF=S_{\cF_{vec}}$ and such that the norms of the vectors in $\cF_{vec}$ are given by the weights:
	$$(\frac{\alpha_1}{n}\uno_n\coma\frac{\alpha_2}{n}\uno_n\coma\cdots\coma \frac{\alpha_m}{n}\uno_n)\,.$$
	In particular, the computation of the distance can be done using the results in \cite{MRiS1} (or the results in previous section, for the particular case of $\dd=d$). Notice that, as it was done in the proof of the Proposition \ref{caracterizacion de operadores en Lambda alfa}, optimal G-frames can be constructed from optimal vector frames. 
\end{proof}
\begin{rem} The problem considered in \cite{MLX}, is actually solved in terms of unitarily invariant norms (briefly uin).
	Recall that a norm $N(\cdot)$ in $\mat$ is \textit{unitarily invariant} if 
	$$ 
	\nui{UAV}=\nui{A} \peso{for every} A\in\mat \py U,\,V\in\matud\ ,
	$$
	and $N(\cdot)$ is \textit{strictly convex} if its restriction to diagonal matrices is a strictly convex  norm in $\C^d$. 
	Examples of uin are the spectral norm  and the $p$-norms, for $p\geq 1$
	(strictly convex if $p>1$). Note that, in particular, when $p=2$, we get the Frobenius norm. Then, the problem posed in \cite{MLX} is:
	
	\pausa Given $N(\cdot)$ an strictly convex uin in $\mat$, $\cH=\C^d$, $\cK=\C^n$, $\al=(\al_i)_{i\in\I_m}$ and $A\in\matpos$, compute 
	$$\min_{\cF\in \Lambda_\alpha} N\paren{A-S_\cF},$$
	and characterize the G-Bessel sequences that reach the
	minimum distance. Following the same steps as for the Frobenius norm, and applying Theorem 4.1 in \cite{MRiS}, we get the following generalization of Theorem \ref{g-frame global}, since the minimizers do not depend on the  unitary invariant norm chosen. 
\EOE	
\end{rem}
\begin{teo}\label{g-frame global nui} 
	Let $\alpha=(\alpha_j)_{j\in\I_m}(\R_{\geq 0}^m)\da$, $N(\cdot)$ an strictly convex uin in $\mat$ and consider the set of G-frames $\Lambda_\alpha$ as before. Let $A\in \cM_{d}(\C)^+$. Then, there exists
	$\cF^{\rm op}=\{T_j^{\rm op}\}_{j\in \I_m}\in \Lambda_\alpha$
	such that 
	\beq
	N(A- S_{\cF^{\rm op}})\leq N(A- S_{\cF})\quad\text{for all}\quad \cF\in \Lambda_\alpha\,.
	\eeq
	Moreover, the minimal distance (and the approximants)  can be computed using the spectrum of $A$ and the weights  $(\frac{\alpha_1}{n}\uno_n\coma\frac{\alpha_2}{n}\uno_n\coma\cdots\coma \frac{\alpha_m}{n}\uno_n)$ as the initial data for the classical approximation problem. Even more, the best approximants do not depend on the choice of the strictly convex uin.
	\qed
\end{teo}

\end{document}